\documentclass[12pt]{amsart}
\usepackage{amsmath,amsthm,amsfonts,amssymb,latexsym,mathrsfs}
\usepackage[top=1in, bottom=1in, left=1in, right=1in]{geometry}
\usepackage{graphics,color}
\usepackage{bm}
\usepackage{times}
\usepackage{mathtools}
\usepackage{tkz-euclide}
\usepackage{array,tabularx}
\usepackage{pstricks-add, pst-eucl}
\usepackage{auto-pst-pdf}
\usepackage{tabularx,booktabs}
\usepackage{graphicx}
\usepackage{xspace}
\usepackage{hyperref}
\usepackage{hhline}
\usepackage{mathtools}
\usepackage{parskip}
\usepackage{thmtools,thm-restate}

\definecolor{myblue}{rgb}{0.09,0.32,0.44} 
\hypersetup{pdfborder={0 0 0},pdfborderstyle={},colorlinks=true,linkcolor=myblue,citecolor=myblue,
urlcolor=blue}

\usepackage[shortlabels]{enumitem}

\usepackage[utf8]{inputenc}

\usepackage{upgreek}
\usepackage{enumitem}

\usepackage{graphicx,tikz}
\usetikzlibrary{shapes.geometric, arrows}

\usepackage [autostyle, english = american]{csquotes}
\MakeOuterQuote{"}
\usepackage{amsthm}
\usepackage{thm-restate}

\numberwithin{equation}{section}
\newtheorem{theorem}{Theorem}[section]
\newtheorem{Lemma}[theorem]{Lemma}

\newtheorem{prop}[theorem]{Proposition}

\newtheorem{remark}[theorem]{Remark}

{}

\theoremstyle{remark}

\theoremstyle{remark}

\theoremstyle{definition}

\usepackage[utf8]{inputenc}

\begin{document}
\title[POLYNOMIALS WITH A GIVEN NUMBER OF IRREDUCIBLE FACTORS]{ON THE DISTRIBUTION OF POLYNOMIALS HAVING A GIVEN NUMBER OF IRREDUCIBLE FACTORS OVER FINITE FIELDS}
\author[Arghya Datta]{Arghya Datta}
\address{D\'epartement de math\'ematiques et de statistique\\
Universit\'e de Montr\'eal\\CP 6192 Succ. Centre-ville, Montr\'eal, QC H3C 3J7, Canada}
\email{\href{arghyadatta8@gmail.com}{\nolinkurl{arghya.datta@umontreal.ca}}}

\date{}
\begin{abstract}
Let $q\geqslant 2$ be a fixed prime power. We prove an asymptotic formula for counting the number of monic polynomials that are of degree $n$ and have exactly $k$ irreducible factors over the finite field $\mathbb{F}_q$. We also compare our results with the analogous existing ones in the integer case, where one studies all the natural numbers up to $x$ with exactly $k$ prime factors. In particular, we show that the number of monic polynomials grows at a surprisingly higher rate when $k$ is a little larger than $\log n$ than what one would speculate from looking at the integer case. 
\end{abstract}

\maketitle

\section{Introduction}

How many integers are there up to $1000$ billion (or any large number) with exactly 2022 (or any given integer) prime factors? This exciting question fascinated mathematicians for ages. Gauss initially started the investigation and the question was later pursued by several eminent mathematicians, including Landau, Hardy and Ramanujan, Sathe and Selberg. Thanks to their compound effort, we entirely understand these answers in the integer setting. The main goal of this work is to study the analogous question for polynomials over finite fields. We begin our discussion with a brief historical interlude and then state our main results.\\
Let $m\leqslant x$ be positive integers and $\Omega(m)$ denotes the number of prime factors counted with multiplicity that divide $m$. In 1900, Landau [\ref{l30}] showed using the prime number theorem that for $k\in\mathbb{N}$ fixed we have
$$\Pi(x,k)\sim \dfrac{x}{\log x} \dfrac{(\log\log x)^{k-1}}{(k-1)!},$$
where
\[\Pi(x,k):= \#\bigg \{m\leqslant x:\Omega(m)=k \bigg \}.\]
The above result was predicted by Gauss [\ref{l32}] long ago. However, understanding the asymptotic behaviour of $\Pi(x,k)$ for growing $k$ uniformly with $x$ remained an open problem for a long time.\\
In 1953, Sathe [\ref{l5}] ingeniously showed that if $\epsilon\in(0,2)$ is given, then uniformly for $1\leqslant k\leqslant (2-\epsilon)\log\log x$ we have
\begin{equation}\label{l260}
\Pi(x,k)\sim F\left( \dfrac{k}{\log\log x} \right)\dfrac{{(\log\log x)}^{k-1}}{(k-1)!}\dfrac{x}{\log x},
\end{equation}
where
\[
F(z)=\dfrac{1}{\Gamma(z+1)}\displaystyle\prod\limits_{p} \left(1-\dfrac{1}{p} \right)^z\left(1-\dfrac{z}{p} \right)^{-1}\quad(|z|<2).\]
\\This beautiful result was further extended by Selberg [\ref{l261}] when $k$ is a little larger and varies uniformly in a slightly wider range $(2+\epsilon)\log\log x\leqslant k \leqslant B^{*}\log\log x$, where $B^{*}>2$ is a real constant. In his paper Selberg only discussed the main ideas behind the proof. He observed a significant change in the asymptotic behaviour of $\Pi(x,k)$. He showed for this range we have
\begin{equation}\label{l25}
\Pi(x,k)\sim \dfrac{Cx\log (x/2^k)}{2^k},
\end{equation}
where $C$ is an absolute constant. 
Finally Nicolas [\ref{l4}] settled this question for the entire range $(2+\epsilon)\log\log x\leqslant k< \log x/\log 2$ by proving the same asymptotic result holds as (\ref{l25}) in this range. Unlike the previous attempts his main idea behind the proof was less analytic and more of combinatorial nature.\\
As mentioned earlier, we will be working in the polynomial ring  $\mathbb{F}_q[t]$. Henceforth $q$ is assumed to be fixed and all the implied constants are allowed to depend on $q$. For an easy reference later on, we employ some standard notation that will be used throughout the paper.\\
\subsection{Notation}\label{lc} We set
$$\mathcal{M}:=\bigg\{ \text{set of all monics in } \mathbb{F}_q[t] \bigg \},$$
$$\widetilde{\mathcal{M}}:=\bigg\{ \text{set of all monics in } \mathbb{F}_q[t] \text{ having no root in }\mathbb{F}_q\bigg \},$$
$$\mathcal{P}:=\bigg\{ \text{set of all monic irreducibles in } \mathbb{F}_q[t] \bigg \}.$$
$\mathcal{M}_n$ consists of only those monics that have degree $n$ and similarly we define $\widetilde{\mathcal{M}}_n$, $\mathcal{P}_n$. Let $\mathcal{P}_{\geqslant j}$ denotes the set of all monic irreducibles whose degree is at least $j.$\\
Given $f\in\mathbb{F}_q[t]$, let $\Omega(f)$ denotes the number of irreducible polynomials counted with multiplicity that divide $f$. We also set
$$N(n,k):= \#\bigg \{f\in\mathcal{M}_n:\Omega(f)=k \bigg \},$$
$$N'(n,k):= \#\bigg \{f\in \widetilde{\mathcal{M}_n}:\Omega(f)=k \bigg \}.$$
As discussed initially, we are interested in understanding the object $N(n,k)$ for different ranges of $k$. Surprisingly one of our main results reveals that the situation in $\mathbb{F}_q[t]$ is quite different than what one could naively guess from the knowledge in the integer case. This is remarkable since these two worlds are significantly different in only a few instances.\\
We first give a proof of (\ref{l260}) in the context of polynomials. The result was first proved in Warlimont [\ref{r2}] and Car [\ref{r1}]. The following result has been generalized in a recent work [\ref{ofir}], which is uniform in both parameters $n$ and $q$. However, $q$ remains fixed for our purposes and we sketch a proof of Warlimont's result for completeness and the techniques we discuss are tailored in a way that will naturally extend to our main result which is Theorem \ref{l10}.\\
\begin{theorem}\label{l6}
Given $\epsilon\in(0,q)$, then uniformly for $1\leqslant k\leqslant (q-\epsilon)\log n$ we have
   
   $$N(n,k)\sim \dfrac{q^n}{n}H\left( \dfrac{k}{\log n} \right)\dfrac{{(\log n)}^{k-1}}{(k-1)!},$$
where \[H(z)=\dfrac{1}{\Gamma(z+1)} \displaystyle\prod\limits_{p\in\mathcal{P}} \left(1-\dfrac{1}{q^{deg(p)}} \right)^z\left(1-\dfrac{z}{{q^{deg(p)}}} \right)^{-1}\] for $|z|<q.$
\end{theorem}
We may notice that the expression for the function $H$ is quite similar to $F$ in the integer case. The appearance of $q^n/n$ is also no surprise as letting $q^n\approx x$ yields $q^n/n \approx x/\log x$ up to constant factors. While the above result is more or less an application of the Selberg-Delange method developed in [\ref{l1}], the following surprising outcome of our work is where a substantial amount of new techniques goes in.\\
We would like to mention that Hwang [\ref{r3}] considered the object $N(n,k)$ for the entire range $1\leqslant k\leqslant n$. We reproduce similar results but using a completely different approach. We consider a mild restriction on the range of uniformity of $k$ but this assumption allows us to obtain a result that has a considerably simpler shape (only in terms of elementary functions) compared to Hwang. Hwang considered the following decomposition $\Omega(f)=\Omega^{(1)}(f)+\Omega^{(2)}(f)$, where $\Omega^{(1)}(f)$ and $\Omega^{(2)}(f)$ denote the number of monic irreducible factors of $f$ of degree $\geqslant 2$ and $=1$. Roughly speaking, Hwang showed that if $f\in\mathcal{M}_n$ is chosen uniformly at random then the object $\Omega^{(1)}$ is roughly Poisson distributed while $\Omega^{(2)}$ follows a negative binomial distribution. Hwang then looks at the sum of these two random objects and studies their convolution law for different ranges of $k$. He showed that if $k$ is small, i.e. $1\leqslant k\leqslant (q-\epsilon)\log n$, then the behaviour of $\Omega_q=\Omega_q^{(1)}+\Omega_q^{(2)}$ is dominated by the $\Omega_q^{(1)}$ component and thus resulting an approx Poisson distribution (as obtained by Warlimont earlier)
$$\mathbb{P}(\Omega_q=k)\sim H\left(\dfrac{k}{\log n}\right)\dfrac{(\log n)^{k-1}}{(k-1)!}\dfrac{1}{n}.$$
On the other hand, if $k$ is "largish" and in particular $k/\log n\to\infty$, then the sum $\Omega_q$ is considerably inflated by the component $\Omega_q^{(2)}$ resulting a significant change in the final asymptotic. Contrary to Hwang's probabilistic approach, our line of attack uses a crucial combinatorial decomposition identity from [\ref{l4}] in the context of polynomials.\\
It is worthwhile to mention that similar studies concerning the function $\omega(f)$ instead of $\Omega(f)$ can be found in [\ref{r5}].
\begin{theorem}\label{l10}
Let $B\geqslant2$ be a real constant and $q>2$ be a prime power. Let $\xi:\mathbb{N}\to\mathbb{R}$ be a function such that $\xi(n)\to\infty$ however slowly as $n\to\infty$. Then uniformly for $\xi(n)\log n\leqslant k\leqslant {n}/{B}$ we have
$$N(n,k) \thicksim C(q)\dfrac{q^nk^{q-1}(n-k)^{q-1}}{q^k},$$
where  \[C(q)=\dfrac{1}{\left((q-1)!\right)^2}\left( 1-\dfrac{1}{q}\right)^{q^2}\displaystyle\prod\limits_{p\in\mathcal{P}_{\geqslant 2}}\left(1-\frac{1}{q^{deg(p)-1}}\right)^{-1}\left( 1-\frac{1}{q^{deg(p)}}\right)^q,\] a constant that depends only on $q$.  
\end{theorem}
Here ``however slowly'' means that $\xi(n)$ could be any function that grows to infinity with an additional constraint that $\xi$ is sufficiently small so that the interval $[\xi(n)\log n,n/B]$ contains at least one integer for $n$ large enough. We remind the reader that $x/2^k$ is replaced by $q^n/q^k$ as was expected. The appearance of $(n-k)$ is not also surprising in light of the main theorem proved in [\ref{l4}] and (\ref{l25}) where a factor of $\log \left(x/2^k\right)$ is present in the statement. However, the extra factor $k^{q-1}$ and also the presence of higher powers of $(n-k)$ are what seem surprising at the moment, which shows $N(n,k)$ grows at a much larger rate than what we would predict from (\ref{l25}).\\
We note Theorem \ref{l10} is valid for $q\neq 2$. Our methods do not cover the case $q=2$ for technical reasons. We will discuss this technical restriction in the proof of Proposition \ref{l14}. (see Remark \ref{Y}) Also, there is a slight gap between the upper bound of $k$ in Theorem \ref{l6} and the lower bound of $k$ in Theorem \ref{l10}. While it is fairly easy to prove an upper bound of the shape $N(n,k)\ll q^{n-k}k^{q-1}(n-k)^{q-1}$ in the intermediate range $q\log n \leqslant k \ll \log n$ using the ideas in Proposition \ref{l14}, it seems difficult to obtain any meaningful and ``simple to state'' lower bound for that range. The proofs of the above theorems are discussed in detail in the following sections. The main ingredients are the following two key propositions.
\begin{prop}\label{la}
For all $\epsilon\in(0,q)$ and uniformly for $|z|\leqslant q-\epsilon$ we have
$${M_z}(n):=\displaystyle\sum\limits_{f\in {\mathcal{M}_n}}z^{\Omega(f)}=q^nn^{z-1}\Bigg \{zH(z)+O_{\epsilon}\left(\dfrac{1}{n} \right) \Bigg \},$$
with $H(z)$ as described in Theorem $\ref{l6}$.
\end{prop}
\begin{restatable}{prop}{llll}\label{lb}
For all $\delta\in(0,1)$ and uniformly for $|z|\leqslant q^2-\delta$ we have
$$\widetilde{M_z}(n):=\displaystyle\sum\limits_{f\in\widetilde{\mathcal{M}_n}}z^{\Omega(f)}=q^nn^{z-1}\Bigg \{zh(z)+O_{\delta}\left(\dfrac{1}{n} \right) \Bigg \},$$
where $h(z)=\dfrac{1}{\Gamma(z+1)}\left( 1-\dfrac{1}{q}\right)^{qz}\displaystyle\prod\limits_{p\in\mathcal{P}_{\geqslant 2}}\left(1-\frac{z}{q^{deg(p)}}\right)^{-1}\left( 1-\frac{1}{q^{deg(p)}}\right)^z$, an analytic function for $|z|<q^2$.
\end{restatable}
The connection between Theorem $\ref{l6}$ and Proposition $\ref{la}$ becomes clear once we notice $N(n,k)$ is precisely given by the coefficient of $z^k$ in $M_z(n)$. In other words, if we could extract the information about $N(n,k)$ from $M_z(n)$, we would have our result. Furthermore, we note that $M_z(n)$ is polynomial in $z$ of degree at most $n$ and therefore, we could hope to study the coefficients using Cauchy's integral formula. This is going to be our main strategy to attack Theorem $\ref{l6}$. The connection between $\widetilde{M_z}(n)$ and Theorem $\ref{l10}$ is not as immediate as the previous case since the coefficient of $z^k$ in $\widetilde{M_z}(n)$ is $N'(n,k)$ and not $N(n,k)$. As we shall see, in section \ref{ld} the two objects $N(n,k)$ and $N'(n,k)$ are intimately entangled, and information about one could be transferred into another. Also, one can check the relation $C(q)=qh(q)/(q-1)!$ holds.\\

\textbf{Acknowledgements. }The author would like to thank his supervisor, Prof. Andrew Granville, for suggesting that he work on this problem and also his co-supervisor, Prof. Dimitris Koukoulopoulos, for many helpful conversations and continuous support during this project and for their valuable suggestions and comments on the previous versions of this manuscript. The author would also like to thank the anonymous referee for many useful comments and Ofir Gorodetsky for informing him about the previous works on the same problem, which the author was initially unaware of. The author was financially supported by his supervisors and the fellowships provided by Faculté des études supérieures et postdoctorales (FESP), Bourse d'exemption of Université de Montréal, Centre interuniversitaire en calcul mathématique algébrique (CICMA) while carrying out this work.  
\section{Set-up and the proof of Proposition \ref{la}}
\subsection{Preparatory results for Proposition $\ref{la}$}
In this subsection we aim to develop the framework for a Selberg-Delange type argument that will be used to establish Proposition $\ref{la}$. As already discussed we are interested in the quantity
$$M_z(n):=\displaystyle\sum\limits_{f\in\mathcal{M}_n} z^{\Omega(f)}=\displaystyle\sum\limits_{k\geqslant 0}z^k N(n,k).$$
We also want to introduce an auxiliary quantity
\begin{equation}\label{l40}
G_z(u)=\displaystyle\sum\limits_{f\in\mathcal{M}} z^{\Omega(f)}u^{deg (f)}=\displaystyle\prod_{p\in\mathcal{P}}\left(1-zu^{deg(p)}\right)^{-1}.
\end{equation}
The function $G_z(u)$ defines an analytic function since the Euler product converges uniformly on the compacts subsets of the region defined by $|u|<\min \big\{ |z|^{-1},q^{-1}\big \}$ which we argue now.\\ 
We recall a well known result from complex analysis that says for a sequence of complex numbers $\{ a_n\}\subset \mathbb{C}$ with $a_n\neq -1$ if $\sum\limits_{n=1}^{\infty} |a_n|<\infty$, then $$\prod\limits_{n=1}^{\infty} (1+a_n) \text{ converges. }$$
Using this result we see the product in (\ref{l40}) converges uniformly on the compact subsets of the region $0<|u|<q^{-1}.$ Indeed we have
\begin{align}
\nonumber
\displaystyle\sum\limits_{p\in\mathcal{P}}\left|zu^{deg(p)}\right|  &\leqslant \displaystyle\sum\limits_{p\in\mathcal{P}}|u|^{deg(p)-1}  \qquad (|uz|<1)\\
\nonumber
&\leqslant q+ \displaystyle\sum\limits_{p\in\mathcal{P}_{\geqslant 2}}|u|^{deg(p)-1}\\
\nonumber
&\leqslant q+ \displaystyle\sum\limits_{n=2}^{\infty}q^n|u|^{n-1} \qquad \left(\text{ since } \#\mathcal{P}_n\leqslant q^n \right)\\
\nonumber
&\leqslant q+ \dfrac{1}{|u|}\displaystyle\sum\limits_{n=2}^{\infty} |qu|^n< \infty, \qquad (|qu|<1)
\end{align}
where the finite quantity $\sum\limits_{n=2}^{\infty} q^n|u|^{n-1}$ depends only on the chosen compact subset of the region $0<|u|<q^{-1}$, but not on the point $(u,z)$.
Finally we check that \[|1-zu^{deg(p)}|\geqslant 1-|zu^{deg(p)}|\geqslant 1-|zu|>0,\] showing there is no pole in the concerned region.\\
Our job is to understand $M_z(n)$, which is precisely given by the coefficient of $u^n$ in the power series of $G_z(u)$. Hence information about $M_z(n)$ can be obtained from $G_z(u)$ by applying Cauchy's integral formula. We took this leap from $M_z(n)$ to $G_z$ because it seems to have a nicer arithmetic structure (i.e. an Euler product expansion), making it more amenable from the analytic point of view. Therefore it is natural to seek to meromorphically extend $G_z(u)$ beyond the region of absolute convergence discussed above to collect the contribution coming from the singularity of $G_z$ in the extended domain. We do this job by introducing another function $F_z(u)$ via the Euler product
\begin{equation}\label{l55}
F_z(u)=\displaystyle\prod\limits_{p\in\mathcal{P}}\left(1-zu^{deg(p)} \right)^{-1}\left(1-u^{deg(p)} \right)^z.
\end{equation}
We intend to analyze the above Euler product in the following open region and show it is holomorphic.
$$\mathcal{R}:=\Big \{(u,z)\in\mathbb{C}^2:|z|< q,|u|<|z|^{-1},|u|<q^{-1/2}\Big\}\subset\mathbb{C}^2.$$
The following lemma shows $F_z(u)$ does converge uniformly on the compact subsets of $\mathcal{R}$ which is what we wanted to establish.
\begin{Lemma}\label{l56}
We have that $F_z(u)$ defined in (\ref{l55}) converges uniformly on the compact subsets of $\mathcal{R}$ and hence defines a holmorphic function in the sense two complex variables.
\end{Lemma}
\begin{proof}
We take the standard branch of complex logarithm defined over $\mathbb{C}\setminus(-\infty,0]$. Since $|u|^n,|zu^n|<1$ we could make use of their Taylor series and write
\begin{align}
\nonumber
z\log(1-u^n)-\log(1-zu^n)  &= -zu^n+O\left(\left|zu^{2n}\right|\right)+zu^n+O\left(\left|z^2u^{2n}\right|\right)\\
\nonumber
&= O\left(\left|q^2u^{2n}\right|\right)\qquad \left(|z|<q\right)\\
\end{align}
and hence using $\exp(O(x))=1+O(x)$ for $x=O(1)$ we get
$$(1-u^n)^z(1-zu^n)^{-1}=1+O\left(\left|q^2u^{2n}\right|\right).$$
We now observe the product in (\ref{l55}) converges uniformly in the compact subsets of $\mathcal{R}$. Since $\#\mathcal{P}_n\leqslant q^n$, we have that $$\displaystyle\sum\limits_{p\in\mathcal{P}}|u|^{2n}\leqslant \displaystyle\sum\limits_{n\geqslant 1}|qu^2|^{n}<\infty,$$
where the finite quantity $\sum\limits_{n\geqslant 1}|qu^2|^{n}$ depends on the chosen compact subset of $\mathcal{R}$, but is independent of the point $(u,z)$.
\end{proof}
We recall the Riemann zeta function in the context of polynomials over finite field, defined for $|u|<q^{-1}$
$$\zeta(u)=\dfrac{1}{1-qu}=\displaystyle\prod_{p\in\mathcal{P}}\left(1-u^{deg(p)}\right)^{-1}.$$
As we can see $\zeta(u)$ can be continued mermorphically on $\mathbb{C}$ with a simple pole at $u=1/q$. Furthermore considering the Euler products in (\ref{l40}) and $(\ref{l55})$ we can connect $G_z(u)$ and $F_z(u)$ in the following way whenever $|u|<\min \big\{ |z|^{-1},q^{-\frac 12}\big\}$
\begin{equation}\label{l621}
G_z(u)=\zeta(u)^zF_z(u).
\end{equation}
We are now in a position to prove our first Proposition. However, first, we want to record the following estimate [\ref{l1}, Lemma 1], which is going to be helpful in the further calculation of the contribution that comes from the singularity of $G_z(u)$ at $u=1/q$. 
\begin{figure}[h]
    \centering
    
    \caption{The contour $\mathcal{H}$ appearing in Lemma \ref{l3}}
    \label{A}
\tikzset{every picture/.style={line width=0.75pt}} 

\begin{tikzpicture}[x=0.75pt,y=0.75pt,yscale=-1,xscale=1]

\draw    (261,140) -- (261.98,244) ;
\draw [shift={(262,246)}, rotate = 269.46] [color={rgb, 255:red, 0; green, 0; blue, 0 }  ][line width=0.75]    (10.93,-3.29) .. controls (6.95,-1.4) and (3.31,-0.3) .. (0,0) .. controls (3.31,0.3) and (6.95,1.4) .. (10.93,3.29)   ;
\draw    (256,141) -- (390,141) ;
\draw [shift={(392,141)}, rotate = 180] [color={rgb, 255:red, 0; green, 0; blue, 0 }  ][line width=0.75]    (10.93,-3.29) .. controls (6.95,-1.4) and (3.31,-0.3) .. (0,0) .. controls (3.31,0.3) and (6.95,1.4) .. (10.93,3.29)   ;
\draw    (261,140) -- (261,32) ;
\draw [shift={(261,30)}, rotate = 90] [color={rgb, 255:red, 0; green, 0; blue, 0 }  ][line width=0.75]    (10.93,-3.29) .. controls (6.95,-1.4) and (3.31,-0.3) .. (0,0) .. controls (3.31,0.3) and (6.95,1.4) .. (10.93,3.29)   ;
\draw    (256,141) -- (133,140.02) ;
\draw [shift={(131,140)}, rotate = 0.46] [color={rgb, 255:red, 0; green, 0; blue, 0 }  ][line width=0.75]    (10.93,-3.29) .. controls (6.95,-1.4) and (3.31,-0.3) .. (0,0) .. controls (3.31,0.3) and (6.95,1.4) .. (10.93,3.29)   ;
\draw  [draw opacity=0] (229.31,145.59) .. controls (232.13,159.49) and (245.24,170) .. (261,170) .. controls (278.81,170) and (293.25,156.57) .. (293.25,140) .. controls (293.25,123.43) and (278.81,110) .. (261,110) .. controls (244.77,110) and (231.35,121.15) .. (229.09,135.66) -- (261,140) -- cycle ; \draw   (229.31,145.59) .. controls (232.13,159.49) and (245.24,170) .. (261,170) .. controls (278.81,170) and (293.25,156.57) .. (293.25,140) .. controls (293.25,123.43) and (278.81,110) .. (261,110) .. controls (244.77,110) and (231.35,121.15) .. (229.09,135.66) ;  
\draw    (158,145) -- (201.31,145.59) ;
\draw    (157.16,133.76) -- (188.17,134.7) ;
\draw [shift={(190.16,134.76)}, rotate = 181.74] [color={rgb, 255:red, 0; green, 0; blue, 0 }  ][line width=0.75]    (10.93,-3.29) .. controls (6.95,-1.4) and (3.31,-0.3) .. (0,0) .. controls (3.31,0.3) and (6.95,1.4) .. (10.93,3.29)   ;
\draw    (229.31,145.59) -- (203.31,145.59) ;
\draw [shift={(201.31,145.59)}, rotate = 360] [color={rgb, 255:red, 0; green, 0; blue, 0 }  ][line width=0.75]    (10.93,-3.29) .. controls (6.95,-1.4) and (3.31,-0.3) .. (0,0) .. controls (3.31,0.3) and (6.95,1.4) .. (10.93,3.29)   ;
\draw    (190.16,134.76) -- (229.09,135.66) ;

\draw (263,143.4) node [anchor=north west][inner sep=0.75pt]    {$O$};
\draw (379,146.4) node [anchor=north west][inner sep=0.75pt]    {$x$};
\draw (257,14.4) node [anchor=north west][inner sep=0.75pt]    {$y$};
\draw (146,145.4) node [anchor=north west][inner sep=0.75pt]    {$ \begin{array}{l}
-n\delta \\
\end{array}$};

\end{tikzpicture}

\end{figure}
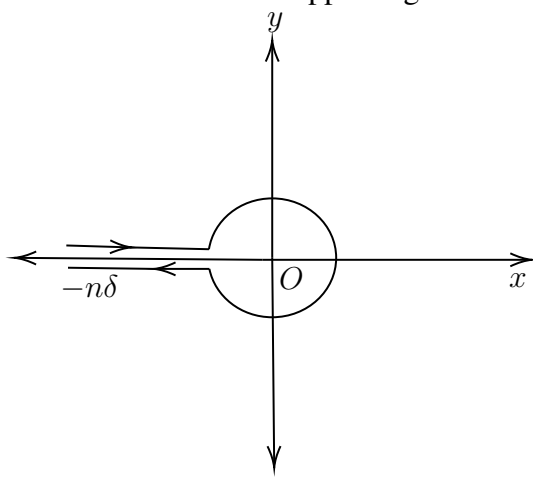
\begin{Lemma}\label{l3}
 Let $A, \delta>0$. Let $\mathcal{H}$ be the Hankel contour of radius $1$ around $0$ going in clockwise direction along the negative real axis to $-n\delta$. Then uniformly for $|z|\leqslant A$ we have that
$$\dfrac{1}{2\pi i}\displaystyle\int_{\mathcal{H}} w^z \dfrac{dw}{\left(1-\dfrac{w}{n}\right)^{n+1}}=-\dfrac{1}{\Gamma(-z)}+O_{\delta,A} \left( \dfrac{1}{n}\right).$$
\end{Lemma}
\begin{figure}[h]
    \centering
    
    \caption{The contour appearing in the proof of Proposition \ref{la}}
    \label{A}
\begin{tikzpicture}[scale=0.8]
  \tkzInit[xmin=-6,ymin=-6,xmax=6,ymax=6]
  \tkzDrawXY[noticks]
  \tkzDefPoint(0,0){L}
   \tkzDefPoint(3,0){O}
   \tkzDefPoint(3.5,.2){B} \tkzDefPoint(3.5,-.2){C}
   \tkzDefPoint(5.8,.2){D} \tkzDefPoint(5.8,-.2){E}  
  \tkzDrawArc[color=red,line width=1pt](O,B)(C) 
  \begin{scope}[decoration={markings,
      mark=at position .5 with {\arrow[scale=2]{>}};}]
    \tkzDrawSegments[postaction={decorate},color=red,line width=1pt](B,D E,C)
  \end{scope}
  \begin{scope}[decoration={markings,
     mark=at position .20 with {\arrow[scale=2]{>}},
     mark=at position .70 with {\arrow[scale=2]{>}};}]
    \tkzDrawArc[postaction={decorate},color=black,line width=1pt](L,D)(E)
  \end{scope}
  \draw[green,thick] (0,0) circle (1.5cm);
  \draw[style=dashed] (0,0) circle (3);
  \draw (4.4,1) node[below left] {$\textcolor{red}{\mathcal{H}'}$};
  \draw (5.4,6.5) node[below left] {$\textcolor{black}{|u|=1/q+\eta}$};
  \draw (2.6,4) node[below left] {$\textcolor{black}{|u|=1/q}$};
  \draw (2.1,2.3) node[below left] {$\textcolor{green}{|u|=r}$};
\end{tikzpicture}
\end{figure}
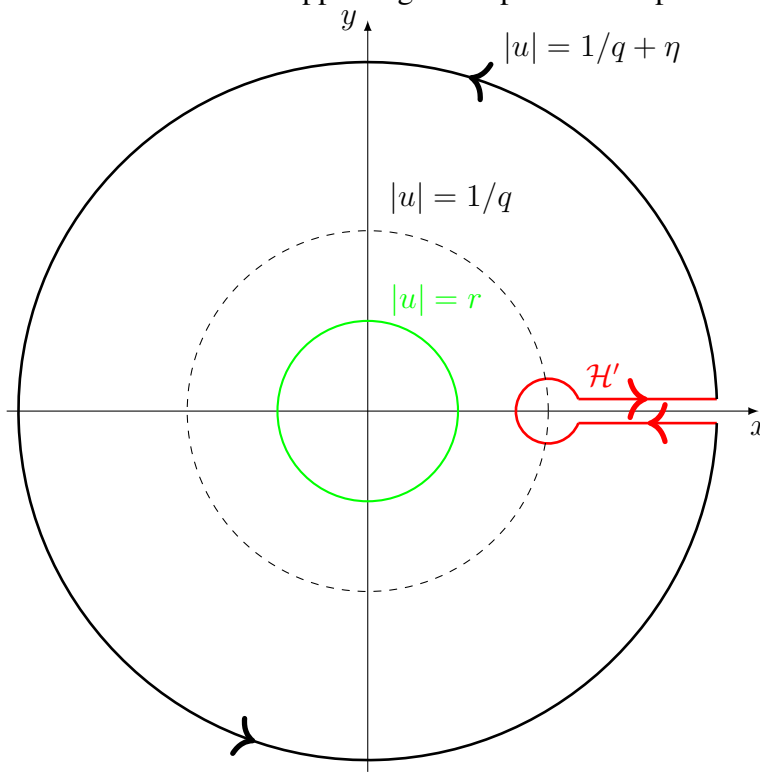
\subsection{Deduction of Proposition \ref{la}}
\begin{proof}[Proof of Proposition \ref{la}]
Let $|u|=r < q^{-1}$ (small green circle in Fig-\ref{A} above). Using Cauchy's integral formula on (\ref{l40}) we get for each fixed $z$ satisfying the hypothesis $|z|\leqslant q-\epsilon$,
\begin{equation}\label{l101}
M_z(n)=\dfrac{1}{2\pi i}\displaystyle\int\limits_{|u|=r}G_z(u) \dfrac{du}{u^{n+1}}=\dfrac{1}{2\pi i}\displaystyle\int\limits_{|u|=r}\zeta(u)^zF_z(u) \dfrac{du}{u^{n+1}}.
\end{equation}
We wish to get past the point $u=1/q$ so that we could collect the contribution coming from the singularity of $\zeta(u)$ there. We shift the smaller (green) circle $|u|=r$ to a bigger circle (black) $|u|=q^{-1}+\eta$ and a portion of Hankel contour (red) around the point $q^{-1}$ (Fig-\ref{A}). (to be specified shortly) Let $\mathcal{H}'$ be the contour that consists of a circle of radius $(qn)^{-1}$ traversed clockwise around $q^{-1}$ and the two line segments on the ray $0$ to $q^{-1}$ joining this small circle to the bigger circle $|u|=q^{-1}+\eta$ where $\eta>0$ is a fixed number. However we must have $q^{-1}+\eta <q^{-\frac 12}$ as we cannot go beyond the region $\mathcal{R}$ in (\ref{l621}) where $G_z(u)$ makes sense. We recall that $|z|\leqslant (q-\epsilon)$ inside $\mathcal{R}$ while performing the above integral. Keeping this in mind we get another constraint on $\eta$, $|uz|\leqslant (q^{-1}+\eta)(q-\epsilon)\leqslant 1$. This shows taking any $0<\eta<\min\{1/(q-\epsilon)-1/q,1/\sqrt{q}-1/q \}$ suffices. We then plan to use Lemma \ref{l3} to evaluate the integral in (\ref{l101}) with $\eta=\min\{1/(q-\epsilon)-1/q,1/\sqrt{q}-1/q \}/2$ for $|z|\leqslant q-\epsilon.$ We observe that for $z$ fixed, $F_z(u)$ and $\zeta(u)$ are both analytic in an open neighbourhood containing the circle $|u|=q^{-1}+\eta$ which is compact. This implies $F_z(u),\zeta(u)$ and hence $G_z(u)$ are uniformly bounded (bound depends on $q,\eta$) on this circle. Thus we obtain
$$M_z(n)=\dfrac{1}{2\pi i}\displaystyle\int\limits_{\mathcal{H}'} G_z(u) \dfrac{du}{u^{n+1}} +O\left( \displaystyle\int\limits_{|u|=\frac{1}{q}+\eta} \left| G_z(u) \dfrac{du}{u^{n+1}}\right|\right):=I+ O_{q,\epsilon}\left( \dfrac{q^n}{(1+q\eta)^n}\right),$$
where at the last step we used the fact that, $\int_{|u|=R}|du/u^{n+1}|= \frac{1}{R^{n+1}}\int_{|u|=R}|du|=2\pi/R^{n}$ and the big-O-uniform bound here depends only on $q$ and $\eta$ while $\eta$ according our choice depends only on $q$ and $\epsilon$.  Now we focus on the integral $I$.\\
Analyticity of $F_z$ near $q^{-1}$ allows us to consider the Taylor series around $q^{-1}$ and write
$$F_z(u)=F_z\left( 1/q\right)+a_1\left(u-1/q\right)+ a_2\left(u-1/q\right)^2+\ldots$$ 
Hence we have that
$$I=F_z\left( \dfrac{1}{q}\right)\dfrac{1}{2\pi i}\displaystyle\int_{\mathcal{H}'}\zeta(u)^z\dfrac{du}{u^{n+1}}+O\left(  \displaystyle\int_{\mathcal{H}'}\left|\left( u-\dfrac{1}{q}\right)\zeta(u)^z \dfrac{du}{u^{n+1}}\right|\right).$$
We use a change of variable so that our contour $\mathcal{H}'$ transforms to $\mathcal{H}$ and we could then use our Lemma \ref{l3}. The change of variable is given by, $w=n(1-uq)\implies u=1/q\left( 1-w/n\right).$ We can quickly check the transformation is indeed a valid one. If $u$ is parameterized by $u=1/q+(1/qn)e^{-it},t\in [\epsilon',2\pi-\epsilon']$  for some small enough $\epsilon'>0$ near the tiny circle around $1/q$ then $w=-e^{-it}$ is a circle revolving around $0$ of radius $1$ in the clockwise direction and similarly the horizontal ray part also gets reflected and $\mathcal{H}'$ indeed transforms into $\mathcal{H}$ and as $|u|$ goes up to $1/q+\eta$, we have $w$ goes up to $-n\delta$ with $\delta=\eta q$. We note for this change of variable $du=-dw/nq$. First we show the alleged error term is small. The $O$-term is 
\begin{align}
\nonumber
\ll_q\displaystyle\int_{\mathcal{H}'}\left|\dfrac{(1-qu)du}{u^{n+1}(1-qu)^z} \right| &=  \displaystyle\int_{\mathcal{H}}\left|\dfrac{1}{\left(\dfrac{w}{n}\right)^{z-1}}\dfrac{\dfrac{-dw}{nq}}{\dfrac{1}{q^{n+1}}\left(1-\dfrac{w}{n} \right)^{n+1}}\right|\\
\nonumber
&=\dfrac{q^n}{n} \displaystyle\int_{\mathcal{H}}\left|\dfrac{1}{\left(\dfrac{w}{n}\right)^{z-1}}\dfrac{dw}{\left(1-\dfrac{w}{n} \right)^{n+1}}\right| \qquad  (\text{ using Lemma }\ref{l3})\\
\nonumber
&\ll q^n n^{\Re(z)-2}.
\end{align}
Now, we evaluate the main term using Lemma \ref{l3}
\begin{align}
\nonumber
F_z\left( \dfrac{1}{q}\right)\dfrac{1}{2\pi i}\displaystyle\int_{\mathcal{H}'}\zeta(u)^z\dfrac{du}{u^{n+1}} &= F_z\left( \dfrac{1}{q}\right)\dfrac{1}{2\pi i}\displaystyle\int_{\mathcal{H}'}\dfrac{1}{(1-qu)^z}\dfrac{du}{u^{n+1}}\\
\nonumber
&=F_z\left( \dfrac{1}{q}\right)\dfrac{q^n}{n}\dfrac{1}{2\pi i}\displaystyle\int_{\mathcal{H}}\dfrac{1}{\left(\dfrac{w}{n}\right)^{z}}\dfrac{-dw}{\left(1-\dfrac{w}{n} \right)^{n+1}}\\
\nonumber
&=F_z\left( \dfrac{1}{q}\right)q^n n^{z-1}\left( \dfrac{1}{\Gamma(z)}+O_{\epsilon}\left( \dfrac{1}{n}\right)\right).
\end{align}
Hence we have that 
\begin{equation}\label{lx}
M_z(n)=F_z\left( \dfrac{1}{q}\right)\dfrac{q^nn^{z-1}}{\Gamma(z)}+O\left( q^n n^{\Re(z)-2}\right)=q^nn^{z-1}\Bigg \{zH(z)+O_{\epsilon}\left(\dfrac{1}{n} \right) \Bigg \},
\end{equation}
where $H(z)=\frac{1}{\Gamma(z+1)}F_z\left( \frac{1}{q}\right)$ is an analytic function for $|z|<q.$ 
\end{proof}
\section{Deduction of Theorem \ref{l6}}
With Proposition $\ref{la}$ at our disposal, we are now ready to prove Theorem $\ref{l6}$. We also note three other straightforward results which will help us do so. We will verify the following claims in the appendix \ref{le}.
 \begin{restatable}{Lemma}{lll}\label{l15}
 For $t\in [-\pi,\pi]$ we have $\cos t-1\leqslant -\dfrac{t^2}{5}$ and $\left|1-e^{it} \right|^2 \leqslant t^2.$
 \end{restatable}
  \begin{restatable}{Lemma}{p}\label{l16} We have the following upper bound
 $$\displaystyle\int_{|z|=r}\left|n^{z-r} \right| |dz|\ll \dfrac{r}{\sqrt{j}},$$
 where $j=r\log n.$
 \end{restatable}
 \begin{restatable}{Lemma}{q}\label{l17} We have the following estimate
$$\displaystyle\int\limits_{-\infty}^{\infty} \theta^2\exp\left( - \dfrac{k\theta^2}{5} \right)d\theta \ll k^{-\frac{3}{2}}.$$
\end{restatable}
\begin{proof}[Deduction of Theorem \ref{l6}]
We plan to recover $N(n,k)$ with another application of Cauchy's integral formula, 
\begin{equation}\label{l26}
N(n,k)=\dfrac{1}{2\pi i}\displaystyle\int\limits_{|z|=\frac{k}{\log n}} M_z(n) \dfrac{dz}{z^{k+1}}.
\end{equation}
Given $\epsilon\in(0,q)$ by the hypothesis of Theorem $\ref{l6}$, from (\ref{lx}) we see the main term of $M_z(n)$ which is $F_z\left( \frac{1}{q}\right)\frac{q^nn^{z-1}}{\Gamma(z)}$ is analytic in the disk $|z|=r={k}/{\log n}\leqslant (q-\epsilon)$ since $F_z\left( \frac{1}{q}\right)$ is analytic for $|z|<q$ and $n^z, 1/\Gamma$ are entire functions. 
Now writing $M_z(n)/q^n= n^{z-1}\left(f(z)+O\left(1/n \right) \right)$ where $f(z)=zH(z)=F_z\left( 1/q\right) \frac{1}{\Gamma(z)}$ and considering the Taylor expansion of $f(z)$ near $z=r$, $f(z)=f(r)+f'(r)(z-r)+O(|z-r|^2)$ we get
\begin{equation}\label{l59}
\dfrac{1}{q^n}N(n,k)=\dfrac{1}{2\pi i}\displaystyle\int_{|z|=r}n^{z-1}f(z) \dfrac{dz}{z^{k+1}}+O(E),
\end{equation}
where $E$ is given by $\dfrac{1}{n}\displaystyle\int_{|z|=r} \left|\dfrac{n^{z-1}}{z^{k+1}}\right||dz|$. We see that
$$\dfrac{1}{2\pi i}\displaystyle\int_{|z|=r} (z-r)n^{z-1}\dfrac{dz}{z^{k+1}}=\dfrac{1}{n}\Bigg\{ \dfrac{(\log n)^{k-1}}{(k-1)!}-r \dfrac{(\log n)^k}{k!}\Bigg \}=0.$$
Hence
\begin{align}
\nonumber
\dfrac{1}{q^n}N(n,k)&= f(r)\dfrac{1}{2\pi i}\displaystyle\int_{|z|=r}n^{z-1} \dfrac{dz}{z^{k+1}}+O\left(r^{-k-1}\displaystyle\int_{|z|=r} |n^{z-1}(z-r)^2||dz| +E \right) \\
\nonumber
&=f(r)\dfrac{(\log n)^k}{k!n}+\text{Error term}.
\end{align}
By using Lemma \ref{l15} and Lemma \ref{l17}, we see the integral in the \text{Error term} is bounded by
$$\dfrac{r^3}{n}\displaystyle\int_{-\pi}^{\pi}|1-e^{i\theta}|^2 e^{k\cos \theta}d\theta \ll \dfrac{r^3}{n}\displaystyle\int_{-\infty}^{\infty} \theta^2 e^{k(1-\frac{\theta^2}{5})}d\theta \ll \dfrac{r^3}{n}e^k k^{-\frac{3}{2}}.$$
Since there is a factor of $r^{-k-1}$ outside, we have the above expression (final step by Stirling's formula)
$$\ll \dfrac{1}{n}r^{2-k}e^kk^{-\frac{3}{2}}\ll (\log n)^{k-2} \dfrac{e^k \sqrt{k}}{nk^k}\ll \dfrac{1}{n}\dfrac{(\log n)^{k-2}}{(k-1)!}.$$
Noting $e^k=n^r\iff k=r\log n$ and using Lemma \ref{l16}, we also see the term $E$ is bounded by
$$\ll \dfrac{r^{-k-1}e^k}{n^2}\displaystyle\int_{|z|=r} |n^{z-r}||dz|\ll \dfrac{r^{-k-1}e^k}{n^2} \left( \dfrac{r}{\sqrt{k}}\right)\ll \dfrac{1}{n^2}\dfrac{(e\log n)^k}{k^{k+\frac{1}{2}}}\ll \dfrac{1}{n^2} \dfrac{(\log n)^k}{k!},$$
where in the last step we have used Stirling. Thus we finally obtain
$$N(n,k)\thicksim q^nf(r) \dfrac{(\log n)^k}{k!n}=\dfrac{q^n}{n}F_r\left( \dfrac{1}{q}\right)\dfrac{1}{\Gamma(r+1)}\dfrac{(\log n)^{k-1}}{(k-1)!}$$
and the proof follows once we recall $r={k}/{\log n}$. The derived expression is identical to the integer case for $k\leqslant (2-\epsilon)\log \log x$ as mentioned in [\ref{l4}] on the first page, a result due to Sathe(1953) [\ref{l5}].\\
\end{proof}
We now begin to discuss the main contribution of this work. For the convenience of the reader, we provide the following outline.
\subsection{Organization of the following sections} The proof of Theorem \ref{l10} rests on a technique developed by Sathe and Selberg along with a combination of an observation made by Nicolas [\ref{l4}] in his work and an analogue of the Selberg-Delange style argument for function fields developed in Porrit [\ref{l1}, section 2]. We will also use ideas from Tenenbaum [\ref{l262}, Chapter II.6], where the author outlines an application of the Selberg-Delange method to integers with a given number of prime factors. In the earlier sections, we gave detailed proofs of Theorem \ref{l6} and Proposition \ref{la} as a direct application of the Selberg-Delange method discussed in [\ref{l1}]. The flow of the argument and preliminary lemmas leading to Theorem $\ref{l10}$ are fairly delicate. Proof of Proposition \ref{lb} will be fundamentally the same as that of Proposition \ref{la} apart from some technical modifications. We will begin section \ref{ld} by establishing Lemma \ref{led}, an odd-even decomposition for polynomials. We will then discuss the role of the auxiliary results Lemma \ref{l8}, Lemma \ref{l11}, Lemma \ref{l12}, Lemma \ref{l13} behind the main proof of the Theorem \ref{l10}. Next, we establish two important results, Proposition \ref{l9} and Proposition \ref{l14} which will finally lead us to Theorem \ref{l10}. Detailed proofs of the technical lemmas and Proposition \ref{lb} will be presented in the appendix. The graphical dependency between the results in section \ref{ld} is described below:\\
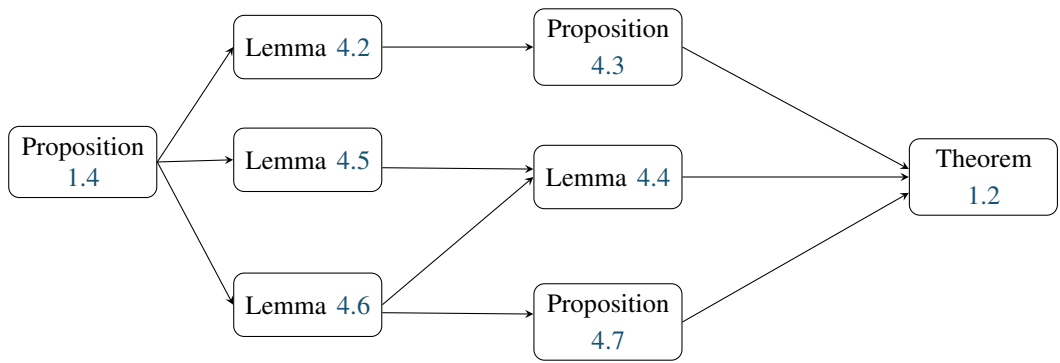
\begin{center}

\makebox[\textwidth]{\parbox{1.5\textwidth}{
			\begin{center}
				\tikzstyle{interface}=[draw, text width=6em,
				text centered, minimum height=2.0em]
				\tikzstyle{daemon}=[draw, text width=4em,
				minimum height=2em, text centered, rounded corners]
				\tikzstyle{lemma}=[draw, text width=6em,
				minimum height=2em, text centered, rounded corners]
				\tikzstyle{dots} = [above, text width=6em, text centered]
				\tikzstyle{wa} = [daemon, text width=6em,
				minimum height=2em, rounded corners]
				\tikzstyle{ur}=[draw, text centered, minimum height=0.01em]
				\def\blockdist{1.3}
				\def\edgedist{0.}
				\begin{tikzpicture}
		\node (thm6)[daemon]  {\footnotesize Lemma \ref{l13} };
		\path (thm6.west)+(-3,-1.7) node (thm2)[daemon] {\footnotesize Lemma \ref{l12}};
		\path (thm6.north) + (0,1.3) node (hyp2)[daemon] {\footnotesize Proposition \ref{l9} };
		\path (thm6.south) + (0,-1.5) node (appT)[daemon] {\footnotesize Proposition \ref{l14} };
	\path (thm6.east)+(4,0) node (thm3)[daemon] {\footnotesize Theorem \ref{l10}};	
\path (thm2.north)+(0,3) node (d1)[daemon] {\footnotesize Lemma \ref{l8} };
\path (thm2.north)+(0, 1.5) node (d3)[daemon] {\footnotesize Lemma \ref{l11}\\};	
\path (thm2.west)+(-2, 1.9) node (d2)[daemon] {\footnotesize Proposition \ref{lb}\\};
\path [draw, ->,>=stealth] (thm6.east) -- node [above] {} (thm3.west) ;
\path [draw, ->,>=stealth] (hyp2.east) -- node [above] {} ([yshift =3]thm3.west); 
\path [draw, ->,>=stealth] (appT.east) -- node [above] {} ([yshift=-6]thm3.west);
\path [draw, ->,>=stealth] (thm2.east) -- node [above] {} (thm6.west);
\path [draw, ->,>=stealth] (d1.east) -- node [above] {} (hyp2.west);
\path [draw, ->,>=stealth] ([yshift=-3]thm2.east) -- node [above] {} ([yshift=3]appT.west);
\path [draw, ->,>=stealth] ([yshift=-3]d3.east) -- node [above] {} ([yshift=3]thm6.west);
\path [draw, ->,>=stealth] (d2.east) -- node [above] {} (d3.west) ;
\path [draw, ->,>=stealth] (d2.east) -- node [above] {} (d1.west) ;
\path [draw, ->,>=stealth] (d2.east) -- node [above] {} (thm2.west) ;

	\end{tikzpicture}
	\end{center}}}
\end{center}
\medskip
\section{Preparatory results and an odd-even decomposition}\label{ld}
\subsection{An odd-even decomposition for polynomials}
We begin this subsection by introducing a decomposition for $\Pi(x,k)$ into $k$ smaller parts used by Nicolas [\ref{l4}] to prove $(\ref{l25})$. We first note that any integer $n\leqslant x$ can be uniquely written as $n=2^j\ell$ with $\ell$ odd. We note the following identity 
\begin{equation}\label{ly}
\Pi(x,k)=\sum\limits_{0\leqslant j\leqslant k}\Pi'(x/2^{k-j},j),
\end{equation}
where \[\Pi'(x,k):= \#\bigg \{m\leqslant x:\text{ $m$ odd and }\Omega(m)=k \bigg \}\]
holds. If $2\log\log x<k<\log x/\log 2$, then the key idea proposed in [\ref{l4}] says the major contribution in \eqref{ly} comes if we restrict $0\leqslant j\leqslant \alpha\log\log (x/2^k)$ for some explicit constant $\alpha>2$.\\
Inspired from the idea above we wish to write $f=hg$ for any $f\in\mathcal{M}_n$, where the polynomials $g$ and $h$ are expected to play the role of ``odd'' and ``even'' respectively in some appropriate sense. Let us define the set
$$\mathcal{S}(j):=\bigg \{ h\in\mathcal{M}_j: h\text{ has only degree }1\text{ irreducible factors }\bigg \}.$$
Given $f\in\mathcal{M}_n$ with $\Omega(f)=k$, we can uniquely write $f=hg$ where $h\in\mathcal{S}(j)$ for some $0\leqslant j\leqslant k$ and $g$ has no degree $1$ irreducible factor. Elements of the set $\mathcal{S}(j)$ mimic the even part while the odd part is mimicked by the polynomial $g$. Using this heuristic we could prove the following decomposition result for $N(n,k)$.
\begin{Lemma}\label{led}
We have \[N(n,k)=\displaystyle\sum_{1\leqslant j\leqslant k} \binom{k-j+q-1}{q-1}N'(n+j-k,j)\] for $1\leqslant k\leqslant n/2.$
\end{Lemma}
\begin{proof}
We observe that for any $h\in \mathcal{S}(j)$, one can find non negative integers $a_1,a_2,\ldots,a_q$ such that $h(t)=(t-\alpha_1)^{a_1}(t-\alpha_2)^{a_2}\ldots(t-\alpha_q)^{a_q}$ with $\Omega(h)=a_1+a_2+\ldots+a_q=j$.
Here $\{\alpha_i\}_{1\leqslant i\leqslant q}$ are the $q$-elements in $\mathbb{F}_q.$ Counting the number of non negative integral solutions we readily see that, $\#\mathcal{S}(j)=\binom{j+q-1}{q-1}.$ If we consider the unique decomposition $f=hg$ for some $f\in\mathcal{M}_n$ with $\Omega(f)=k$ then we note that $j\neq k$. For if $j=k$ then $f=hg$ will imply $f=h$ and hence they both have the same degree, a contradiction since $k\leqslant n/2$ by the hypothesis. Therefore we obtain $$N(n,k)=\displaystyle\sum\limits_{0\leqslant j\leqslant k-1}\displaystyle\sum\limits_{h\in\mathcal{S}(j)}N'(n-j,k-j)=\displaystyle\sum\limits_{0\leqslant j\leqslant k-1}\binom{j+q-1}{q-1}N'(n-j,k-j).$$
The statement of the lemma follows after we make the change of variable $j\mapsto k-j.$
\end{proof}
The rough sketch of our strategy to prove the Theorem \ref{l10} is:
(a) In light of the approach discussed in (\ref{ly}) in the integer setting, we start by splitting the whole sum in the above decomposition up to $j\ll \log(n-k)$ and hence we end up with two parts  so that we can write, $T_1={\sum_{1\leqslant j\ll \log(n-k)}\binom{k-j+q-1}{q-1}N'(n+j-k,j)}$ and $T_2={\sum_{\log(n-k)\ll j\leqslant k}\binom{k-j+q-1}{q-1}N'(n+j-k,j)}$.\\\\
(b) We then show the main contribution comes from $T_1$ and is of order $q^{n-k}k^{q-1}(n-k)^{q-1}$ while the term $T_2$ is significantly smaller and is of order $q^{n-k}(n-k)^{q-2}$. Unlike the ideas discussed in the integer case, which is very combinatorial, our proof will be more analytical. The central idea to use a Selberg-Delange type argument followed by an application of Cauchy's integral formula shall remain the same. The situation over $\mathbb{F}_q[t]$ is quite different due to the presence of many degree $1$ irreducibles which is why an extra binomial term and other technical difficulties arise.\\
We begin our preparation with the easier task, a non-trivial upper bound on $T_2$. To do so, we would like an upper bound on each summand of $T_2$. The following result accomplishes this job. 
\begin{restatable}{Lemma}{r}\label{l8}
Let $0<\delta<1, 0<\eta\leqslant q^2-q-\delta$ be given. We have the following upper bound
$$N'(n,j)\ll_{\delta} q^n n^{q+\eta-1}\left(\dfrac{1}{q+\eta}\right)^j.$$
\end{restatable}
Throughout our discussion, the implied constants are allowed to depend on $q, \delta$ and the value of $\delta$ is soon to be fixed in the following result. It is possible to show the above result as a direct application Proposition $\ref{lb}$. The proof of Proposition $\ref{lb}$ will roughly be in the same spirit as that of Proposition $\ref{la}$. We will present this proof followed by an argument leading to Lemma $\ref{l8}$ in the appendix \ref{lh}, \ref{li}. We now shift our attention to check that the term $T_2$ is small as claimed.
\begin{prop}\label{l9}
Let $q>2$ be an integer and $Y=\log(n-k)$. We have the following upper bound 
$$T_2=\displaystyle\sum\limits_{eqY<j\leqslant k} \binom{q-1+k-j}{q-1}N'(n+j-k,j)\ll_{B} q^{n-k} (n-k)^{q-2}.$$
\end{prop}
\begin{proof}
We plan to use $\displaystyle\binom{\alpha}{\beta}\leqslant {\alpha^\beta}/{\beta!}$ for positive integers $\alpha$ and $\beta$. We apply Lemma \ref{l8} with $\eta=(e-1)q\text{ and }\delta=0.6$ so that $(q+\eta)=eq$. First we check this choice satisfies the hypothesis of Lemma \ref{l8}. We see that $(e-1)q<q^2-q-0.6$ is true whenever $q \geqslant 3$. We now have that
\begin{align}
\nonumber
T_2 &\ll \displaystyle\sum\limits_{j>(q+\eta)Y}\dfrac{(q-1+k-j)^{q-1}}{(q-1)!}q^{n-k+j}(n+j-k)^{q+\eta-1}\left( \dfrac{1}{q+\eta}\right)^j\\
\nonumber
&\ll q^{n-k}\displaystyle\sum\limits_{j>(q+\eta)Y}(q-1+k-j)^{q-1}(n+j-k)^{q+\eta-1}\left( \dfrac{q}{q+\eta}\right)^j\\
\nonumber
&= q^{n-k}\displaystyle\sum\limits_{j>eqY}(n-k)^{q-1}{\left(\dfrac{q-1}{n-k}+\dfrac{k-j}{n-k} \right)^{q-1}}(n-k)^{q+\eta-1}\left( 1+\dfrac{j}{n-k}\right)^{q+\eta-1}e^{-j},
\end{align}
where we used $q+\eta =eq$. Since $k\leqslant n/B$ we have $\left(\frac{q-1}{n-k}+\frac{k-j}{n-k}\right)^{q-1}\leqslant \left(q-1+\frac{k}{n-k}\right)^{q-1}\leqslant \left(q-1+\frac{1}{B-1}\right)^{q-1}\ll_{B}1$ . Also $k+j\leqslant 2k\leqslant n$ gives $\left(1+\frac{j}{n-k}\right)^{eq-1}\ll 1$. Recalling $Y=\log(n-k)$ and continuing from the previous step we get
$$T_2\ll_B q^{n-k}(n-k)^{eq+q-2}\displaystyle\sum\limits_{j>eqY}e^{-j}\ll_B q^{n-k}(n-k)^{eq+q-2}e^{-eqY} = q^{n-k}(n-k)^{q-2}.$$
\end{proof}
We are now only left to estimate $T_1$ and show it dominates in the asymptotic of $N(n,k)$. For the ease of exploration, the calculation of $T_1$ is broken down into three smaller sub-parts. To have a good estimate for $T_1={\sum_{1\leqslant j\ll \log(n-k)}\binom{k-j+q-1}{q-1}N'(n+j-k,j)}$, first it is natural to want to understand the terms $N'(n+j-k,j)$ when the range for $j$ is restricted to $\ll\log (n-k)$. With these observations in mind we could show the following result.
\begin{restatable}{Lemma}{t}\label{l13}
Let $Y=\log(n-k)$. We have uniformly for $j\leqslant eqY,$
$$N'(n+j-k,j)=\dfrac{q^{n+j-k}}{n-k}\Bigg \{ Q_j(Y)+O\left(\dfrac{(\log (n-k))^{j+1}}{j!(n-k)} \right)\Bigg \}$$
where
$$Q_j(X):=\displaystyle\sum\limits_{m+\ell=j-1}\dfrac{1}{m!\ell!}h^{(m)}(0)X^{\ell}.$$
\end{restatable}
If we are given the above result in our toolbox, it is then evident that we are just a few steps away from proving an estimate for $T_1$ once we take care of what happens with the sum when the terms $N'(n+j-k,j)$ are twisted by appropriate binomial coefficients for different $j$'s. However, before we proceed to do that, we would like to mention a few words about what goes into the proof of Lemma $\ref{l13}$. The proof of the above result rests on another two fairly technical auxiliary lemmas, which we shall state next. Details of these three claims are postponed till the appendix \ref{lj}, \ref{lk}, \ref{lm} and can be skipped at the moment as they do not offer many insights into our actual goal. 
\begin{restatable}{Lemma}{s}\label{l11}
We have uniformly for $j\leqslant eq \log n,$
$$N'(n,j)=\dfrac{q^n}{n}\Bigg \{Q_j(\log n)+ O \left( \dfrac{(\log n)^j}{j!n}\right) \Bigg \}.$$
\end{restatable}
\begin{restatable}{Lemma}{u}\label{l12}
Let $m\geqslant 1$ be an integer. We have the following upper bounds
$$h^{(m)}(0)/m!\ll \left( \dfrac{1}{2.9q}\right)^m,$$
$$Q_j(X)\ll \dfrac{X^{j-1}}{(j-1)!}  \text{   and    } Q_j'(X) \ll \dfrac{X^{j-2}}{(j-2)!},$$
for $1\leqslant j\leqslant eq X$ uniformly, where $X$ is a parameter that tends to $\infty$ with $n.$
\end{restatable}
We can now set up the stage for proving Theorem $\ref{l10}$. As mentioned earlier we need to have a precise estimate for the twisted binomial sum appearing in $T_1$. Therefore the following Proposition plays a key role in establishing Theorem $\ref{l10}$.
\begin{prop}\label{l14}
Let $Y=\log(n-k)$. We have the following estimate
$$\displaystyle\sum\limits_{j\leqslant eqY}\displaystyle\binom{q-1+k-j}{q-1}q^j Q_j(Y)=\dfrac{k^{q-1}qe^{qY}}{(q-1)!}h(q)\Big\{1+o(1)\Big \}.$$
\end{prop}
\begin{proof}
We start with an observation first. By the hypothesis of Theorem $\ref{l10}$ we have $k>\xi(n)\log n$ yielding $k-j\geqslant \xi(n)\log n-eq\log (n-k)\geqslant (\xi(n)-eq)\log n\to\infty$. Therefore using the standard fact $\binom{\alpha}{\beta}=(1+o(1))\frac{\alpha^{\beta}}{\beta!}$ where $\beta$ is assumed to be fixed and $\alpha\to\infty$ we obtain 
\begin{align}
\nonumber
\displaystyle\sum\limits_{j\leqslant eqY}\displaystyle\binom{q-1+k-j}{q-1}q^j Q_j(Y) &= \dfrac{1}{(q-1)!}\displaystyle\sum\limits_{j\leqslant eqY}\Big \{q^j(q-1+k-j)^{q-1}Q_j(Y)\Big\}(1+o(1))\\
\nonumber
&=\dfrac{k^{q-1}}{(q-1)!}\displaystyle\sum\limits_{j\leqslant eqY}\Big \{q^jQ_j(Y)\Big\}(1+o(1)),
\end{align}
where we used the the Taylor expansion of $\left(k+q-j-1\right)^{q-1}=k^{q-1}(1+o(1))$ since $j=o(k).$\\\\
We now evaluate the sum
\begin{align}
\nonumber
\displaystyle\sum\limits_{j\leqslant eqY}q^jQ_j(Y) &=\left(\displaystyle\sum\limits_{j\geqslant 1} q^j \sum\limits_{\substack{m+\ell=j-1 \\ m,\ell\geqslant 0}}\dfrac{h^{m}(0)}{m!\ell!}Y^\ell\right)- R\\
\nonumber
 &=\left(q\displaystyle\sum\limits_{\ell\geqslant 0} \dfrac{(qY)^\ell}{\ell!} \sum\limits_{j-m=\ell+1}\dfrac{h^{m}(0)}{m!}q^{j-\ell-1}\right)-R\qquad (\text{after interchanging the sum})\\
 \nonumber
 &=\left(q\displaystyle\sum\limits_{\ell\geqslant 0}\dfrac{(qY)^\ell}{\ell!} \sum\limits_{m\geqslant 0}\dfrac{h^{m}(0)}{m!}q^{m}\right)-R\qquad (\text{since }m=j-\ell-1\text{ and } j\geqslant \ell+1)\\
\nonumber
&=qe^{qY}h(q)-R,\qquad (\text{$h(z)$ is analytic for $|z|<q^2$})
\end{align}
where
\begin{align}
\nonumber
R &= \sum\limits_{j>\lfloor eqY\rfloor}q^j Q_j(Y)\quad (\text{where }\lfloor x\rfloor \text{ denotes the greatest integer }\leqslant x)\\
\nonumber
&= \displaystyle\sum\limits_{j\geqslant\lfloor eqY\rfloor+1} q^j \sum\limits_{\substack{m+\ell=j-1 \\ m,\ell\geqslant 0}}\dfrac{h^{m}(0)}{m!\ell!}Y^\ell\\
\nonumber
&= q\left(\displaystyle\sum\limits_{0\leqslant \ell\leqslant\lfloor eqY\rfloor} \dfrac{(qY)^\ell}{\ell!} \sum\limits_{m\geqslant\lfloor eqY\rfloor-\ell}\dfrac{h^{m}(0)}{m!}q^m\right)+q\left(\displaystyle\sum\limits_{\ell\geqslant\lfloor eqY\rfloor+1} \dfrac{(qY)^\ell}{\ell!} \sum\limits_{m\geqslant 0}\dfrac{h^{m}(0)}{m!}q^m\right)\\
\nonumber
&\ll_q S_1+S_2. \quad(\text{say})
\end{align}
We have the following estimates,
\begin{align}
\nonumber
S_1 &\leqslant \displaystyle\sum\limits_{0\leqslant \ell\leqslant\lfloor eqY\rfloor} \dfrac{(qY)^\ell}{\ell!} \sum\limits_{m\geqslant\lfloor eqY\rfloor-\ell}\left(\dfrac{1}{2}\right)^m \quad \left(\text{using }h^{(m)}(0)/m!<(1/2q)^m\text{ from Lemma \ref{l12}}\right)\\
\nonumber
&=\displaystyle\sum\limits_{0\leqslant \ell\leqslant\lfloor eqY\rfloor} \dfrac{(qY)^\ell}{\ell!}\left(\dfrac{1}{2}\right)^{\lfloor eqY\rfloor-\ell}\\
\nonumber
&=  \left(\dfrac{1}{2}\right)^{\lfloor eqY \rfloor}\displaystyle\sum\limits_{0\leqslant \ell\leqslant\lfloor eqY\rfloor} \dfrac{(2qY)^\ell}{\ell!}\\
\nonumber
&\leqslant e^{0.12qY}.\quad \left(\text{since }\left(\dfrac{1}{2}\right)^{\lfloor eqY \rfloor}=e^{-1.88qY}\text{ and the rest is bounded by the complete sum }e^{2qY}\right)
\end{align}
Similarly
\begin{align}
\nonumber
S_2 &=h(q) \displaystyle\sum\limits_{\ell\geqslant\lfloor eqY\rfloor+1} \dfrac{(qY)^\ell}{\ell!}\\
\nonumber
&\ll (qY)^{\lfloor eqY\rfloor}\displaystyle\sum\limits_{\ell\geqslant 1} \dfrac{(qY)^{\ell}}{\left(\ell+\lfloor eqY \rfloor\right)!}\quad\left(\text{dropping $h(q)$and using a change of variable }\ell\mapsto \ell-\lfloor eqY \rfloor\right)\\
\nonumber
&\leqslant \dfrac{(qY)^{\lfloor eqY\rfloor}}{\lfloor eqY \rfloor!}\displaystyle\sum\limits_{\ell\geqslant 1} \dfrac{(qY)^\ell}{\ell!} \left(\text{since } \displaystyle\binom{\ell+\lfloor eqY \rfloor}{\ell}\geqslant 1\text{ so that } \dfrac{1}{(\ell+\lfloor eqY \rfloor)!}\leqslant \dfrac{1}{\ell!\lfloor eqY \rfloor!}\right)\\
\nonumber
&\ll \dfrac{1}{\sqrt{\lfloor eqY \rfloor}}e^{qY},
\end{align}
where at the last step we used Stirling. Indeed letting $p=eqY $, we have
$$\dfrac{(qY)^{\lfloor eqY\rfloor}}{\lfloor eqY \rfloor!}\ll \dfrac{\left( \dfrac{p}{e}\right)^{\lfloor p\rfloor}}{\left( \dfrac{\lfloor p \rfloor}{e}\right)^{\lfloor p \rfloor}\sqrt{\lfloor p \rfloor}}=\left( \dfrac{p}{\lfloor p \rfloor}\right)^{\lfloor p \rfloor}\dfrac{1}{\sqrt{\lfloor p \rfloor}}\leqslant \left(1+\dfrac{1}{\lfloor p \rfloor} \right)^{\lfloor p \rfloor}\dfrac{1}{\sqrt{\lfloor p \rfloor}}\leqslant \dfrac{e}{\sqrt{\lfloor p \rfloor}}.$$
\end{proof}
\begin{remark}\label{Y}
In the hypothesis of the above proposition, the factor $e$ in the parameter $eqY$ was chosen carefully. The final application of Stirling will not work if we choose a factor strictly smaller than $e$. This choice induces a restriction in the hypothesis of the Lemma \ref{l11}. The proof of the lemma in appendix \ref{X} crucially uses the technical condition $q>e$. This is why our methods do not cover the case $q=2$.
\end{remark}
\begin{proof}[Proof of Theorem \ref{l10}]
\label{Proof of Theorem 3.1}
We have that
\begin{align}
\nonumber
N(n,k) &= \displaystyle\sum_{1\leqslant j\leqslant eq Y} \binom{k-j+q-1}{q-1}N'(n+j-k,j)+\displaystyle\sum_{j> eq Y} \binom{k-j+q-1}{q-1}N'(n+j-k,j)\\
\nonumber
&=T_1+T_2.
\end{align}
Proposition \ref{l9} shows the tail $T_2$ is small and is $\ll q^{n-k}(n-k)^{q-2}.$\\
Lemma \ref{l13} and Proposition \ref{l14} implies,
\begin{align}
\nonumber
T_1 &=\displaystyle\sum\limits_{j\leqslant eqY}\binom{k-j+q-1}{q-1}\dfrac{q^{n+j-k}}{n-k}Q_j(Y)+O\left(\displaystyle\sum\limits_{j\leqslant eqY}\binom{k-j+q-1}{q-1}\dfrac{(\log (n-k))^{j+1}}{j!(n-k)} \right)\\
\nonumber
&=\dfrac{q^{n-k}}{n-k}\displaystyle\sum\limits_{j\leqslant eqY}\binom{k-j+q-1}{q-1}q^jQ_j(Y)+O\left(\displaystyle\sum\limits_{j\leqslant eqY}\binom{k-j+q-1}{q-1}\dfrac{(\log (n-k))^{j+1}}{j!(n-k)} \right)\\
\nonumber
&= \dfrac{q^{n-k+1}k^{q-1}e^{qY}}{(q-1)!(n-k)}h(q)\Big\{1+o(1)\Big \}+O\left((1+o(1))\frac{k^{q-1}}{(q-1)!}\displaystyle\sum\limits_{j=1}^{\infty}\dfrac{(\log (n-k))^{j+1}}{j!(n-k)} \right)\\
\nonumber
&= \dfrac{q^{n-k+1}k^{q-1}e^{qY}}{(q-1)!(n-k)}h(q)\Big\{1+o(1)\Big \}+O\left(k^{q-1}\dfrac{\log(n-k)e^{\log(n-k)}}{n-k} \right)\\
\nonumber
&= \dfrac{q^{n-k+1}k^{q-1}e^{qY}}{(q-1)!(n-k)}h(q)\Big\{1+o(1)\Big \}+O\left(k^{q-1}\log(n-k)\right).
\end{align}
Thus we obtain
$$T_1\thicksim \dfrac{q^{n-k}k^{q-1}e^{qY}}{(n-k)}\dfrac{qh(q)}{(q-1)!}\thicksim C(q)\dfrac{q^nk^{q-1}(n-k)^{q-1}}{q^k},$$
where we used $e^{qY}/(n-k)=(n-k)^{q-1}$ when $Y=\log(n-k)$ and
\[C(q)=\dfrac{qh(q)}{(q-1)!}=\dfrac{1}{\left((q-1)!\right)^2}\left( 1-\dfrac{1}{q}\right)^{q^2}\displaystyle\prod\limits_{p\in\mathcal{P}_{\geqslant 2}}\left(1-\frac{1}{q^{deg(p)-1}}\right)^{-1}\left( 1-\frac{1}{q^{deg(p)}}\right)^q,\] a constant that only depends on $q$.
\end{proof}
\appendix
\section{Auxiliary results towards theorem $\ref{l6}$}\label{le}
We establish the technical Lemmas $\ref{l15}$, $\ref{l16}$ and $\ref{l17}$, used in the proofs of Theorem $\ref{l6}$ and will again be used in showing Proposition $\ref{lb}$.
\lll*
\begin{proof}[Proof of Lemma \ref{l15}]
\label{Proof of Lemma 3.1}
 We define the function $f(t):=1-\frac{t^2}{5}-\cos t$ for $t\in[0,\pi]$ and observe that $f'(t)=0\implies \sin t=\frac{2t}{5}$. A quick graph plot or plugging this equation into a computer reveals that it has only two solutions in the interval $[0,\pi]$ , namely $t=0$ and $t\approx 2.125$. We verify that $f'' (0)=-\frac{2}{5}+\cos 0=\frac{3}{5}>0$ and hence $t=0$ is the global minima in the interval $[0,\pi]$ since $0=f(0)< f(2.125)$, which can also be checked on a computer. Finally since $\cos t$ and $t^2$ are both even functions, we have for all $t\in[-\pi,\pi]$ the following inequality
$$f(t)\geqslant 0 \implies 1-\dfrac{t^2}{5}-\cos t\geqslant 0 \implies \cos t -1\leqslant -\dfrac{t^2}{5}.$$For $0\leqslant x,t\leqslant \pi$ it is well known that, 
$$\sin x\leqslant x\implies \displaystyle\int_{0}^{t} \sin x dx \leqslant \displaystyle\int_{0}^{t}x dx  \implies 1-\cos t\leqslant \dfrac{t^2}{2}.$$
Again the above inequality is true for $-\pi\leqslant t\leqslant \pi$ since $\cos$ and $t^2$ are both even functions. Finally we observe $\left|1-e^{it}\right|^2=(1-\cos t)^2+\sin^2t=2(1-\cos t)\leqslant t^2$ and the claim follows.
\end{proof}
\p*
\begin{proof}[Proof of Lemma \ref{l16}]
Let $I_0=\displaystyle\int_{|z|=r}\left|n^{z-r} \right| |dz|$. After substituting $z=re^{2\pi i t}$ where $t\in \left[-\dfrac{1}{2},\dfrac{1}{2} \right)$ we get
\begin{align}
\nonumber
I_0 &\ll \displaystyle\int_{-\frac{1}{2}}^{\frac{1}{2}}\left|n^{re^{2\pi i t}-r} \right| rdt\\
\nonumber
&= r\displaystyle\int_{-\frac{1}{2}}^{\frac{1}{2}}\left|n^{r(\cos 2\pi t -1)} \right|dt \qquad \left(\left| n^{i r \sin 2\pi t}\right|=1\right)\\
\nonumber
&\ll r\displaystyle\int_{-\frac{1}{2}}^{\frac{1}{2}}n^{-\frac{4\pi^2rt^2}{5}} dt \qquad (\text{ by Lemma } \ref{l15})\\
\nonumber
&\ll r\displaystyle\int_{-\infty}^{\infty}e^{-t^2} \dfrac{dt}{\sqrt{r\log n}}\qquad \left( t\mapsto 2\pi t\sqrt{\dfrac{r\log n}{{5}}} \right)\\
\nonumber
&\ll r\dfrac{1}{\sqrt{r\log n}}=\dfrac{r}{\sqrt{j}}.\qquad (j=r\log n)
\end{align}
\end{proof}
\q*
\begin{proof}[Proof of Lemma \ref{l17}]
\label{Proof of Lemma 3.3}
\begin{align}
\nonumber
\displaystyle\int\limits_{-\infty}^{\infty} \theta^2\exp\left( - \dfrac{k\theta^2}{5} \right)d\theta  &= \dfrac{5\sqrt{5}}{k^{\frac{3}{2}}}\displaystyle\int_{0}^{\infty}\sqrt{z}e^{-z}dz \qquad \left(z=\dfrac{k\theta^2}{5}\right)\\
\nonumber
&\ll k^{-\frac{3}{2}}. \qquad \left( \text{ since }\displaystyle\int_{0}^{\infty}\sqrt{z}e^{-z}dz \text{ is convergent} \right)
\end{align}
\end{proof}
\section{towards Proposition \ref{lb} and auxiliary results for theorem \ref{l10}}\label{lf}
\subsection{A Selberg-Delange style argument for $\widetilde{M_z}(n)$}
\llll*
\begin{proof}[Proof of Proposition \ref{lb}]\label{lh}
\label{Proof of Proposition 1.4}
We recall some definitions for convenience,
$$\widetilde{\mathcal{M}}:=\bigg\{ \textit{set of all monics in } \mathbb{F}_q[t] \textit{ having no root in }\mathbb{F}_q\bigg \},$$
$$\widetilde{\mathcal{M}_n}:=\bigg\{ \textit{set of all of degree }n\textit{ monics in } \mathbb{F}_q[t] \textit{ having no root in }\mathbb{F}_q\bigg \},$$
$$N'(n,k):= \#\bigg \{f\in \widetilde{\mathcal{M}_n}:\Omega(f)=k \bigg \}.$$
We proceed as the proof of Theorem \ref{l6} by first defining the following quantities,
$$\widetilde{M}_z(n)=\displaystyle\sum\limits_{f\in\widetilde{\mathcal{M}_n}}z^{\Omega(f)}=\displaystyle\sum\limits_{k\geqslant 0} z^kN'(n,k) \text{    and   } \widetilde{G}_z(u)=\displaystyle\sum\limits_{f\in\widetilde{\mathcal{M}}}z^{\Omega(f)}u^{deg(f)}.$$
Our Euler product decomposition for $\widetilde{G}_z(u)$ now becomes
$$\widetilde{G_z}(u)=\displaystyle\prod_{p\in\mathcal{P}_{\geqslant 2}}\left(1-zu^{deg(p)}\right)^{-1}.$$
Following the same spirit we now have
$$\widetilde{G}_z(u)=\zeta(u)^z \widetilde{F}_z(u),$$
where $\widetilde{F}_z$ is given by the Euler product below
$$\widetilde{F}_z(u)=(1-u)^{qz}\displaystyle\prod\limits_{p\in\mathcal{P}_{\geqslant 2}}\left(1-zu^{deg(p)} \right)^{-1}\left(1-u^{deg(p)}\right)^z.$$
We first show the above Euler product converges absolutely in the region defined by,
$$\widetilde{\mathcal{R}}:=\Big \{(u,z)\in\mathbb{C}^2:|z|< q^2,|u|^2<|z|^{-1},|u|<q^{-1/2}\Big\}\subset\mathbb{C}^2.$$
We take the standard branch of complex logarithm defined over $\mathbb{C}\setminus(-\infty,0]$ and get for each integer $n\geqslant 2$, (since  $|u|^n<1 ,|zu^n|\leqslant |zu^2|<1)$
\begin{align}
\nonumber
z\log(1-u^n)-\log(1-zu^n)  &= -zu^n+O\left(\left|zu^{2n}\right|\right)+zu^n+O\left(\left|z^2u^{2n}\right|\right)\\
\nonumber
&= O\left(\left|q^4u^{2n}\right|\right)\qquad \left(|z|<q^2\right).\\
\end{align}
and hence using $\exp(O(x))=1+O(x)$ for $x=O(1)$ we get,
$$(1-u^n)^z(1-zu^n)^{-1}=1+O\left(\left|q^4u^{2n}\right|\right).$$
We now observe the Euler product for $\widetilde{F}_z$ converges uniformly on the compact subsets of $\widetilde{\mathcal{R}}$ as
$$\displaystyle\sum\limits_{p\in\mathcal{P}_{\geqslant 2}}|u|^{2n}\leqslant \displaystyle\sum\limits_{n\geqslant 2}|qu^2|^{n}<\infty.\qquad (\text {where we used}, \# \mathcal{P}_n\leqslant q^n)$$
We also verify $|1-zu^{deg(p)}|\geqslant 1-|zu^{deg(p)}|\geqslant 1-|zu^2|>0$, which shows none of the $\left(1-zu^{deg(p)} \right)^{-1}$ factors contribute to a pole in $\widetilde{\mathcal{R}}$.\\
Also $(1-u)^{qz}$ is holomorphic in $\widetilde{\mathcal{R}}$ and this proves $\widetilde{F_z}(u)$ is holomorphic inside $\widetilde{\mathcal{R}}$.
We determine the coefficient of $u^n$ in $\widetilde{{G}_z}(u)$ using Cauchy's integral formula and recover $\widetilde{{M}_z}(n)$ exactly as in the proof of Theorem \ref{l6}. We have that, if $|u|=r < q^{-1}$ (refer to the small green circle in the same Fig-\ref{Ax} below), then using Cauchy's residue formula on (\ref{l40}) gives
\begin{equation}\label{lg}
\widetilde{M_z}(n)=\dfrac{1}{2\pi i}\displaystyle\int\limits_{|u|=r}\widetilde{G_z}(u) \dfrac{du}{u^{n+1}}=\dfrac{1}{2\pi i}\displaystyle\int\limits_{|u|=r}\zeta(u)^z\widetilde{F_z}(u) \dfrac{du}{u^{n+1}}.
\end{equation}
\begin{figure}[h]
    \centering
    
    \caption{The contour appearing in the proof of Proposition \ref{lb}}
    \label{Ax}
\begin{tikzpicture}[scale=0.8]
  \tkzInit[xmin=-6,ymin=-6,xmax=6,ymax=6]
  \tkzDrawXY[noticks]
  \tkzDefPoint(0,0){L}
   \tkzDefPoint(3,0){O}
   \tkzDefPoint(3.5,.2){B} \tkzDefPoint(3.5,-.2){C}
   \tkzDefPoint(5.8,.2){D} \tkzDefPoint(5.8,-.2){E}  
  \tkzDrawArc[color=red,line width=1pt](O,B)(C) 
  \begin{scope}[decoration={markings,
      mark=at position .5 with {\arrow[scale=2]{>}};}]
    \tkzDrawSegments[postaction={decorate},color=red,line width=1pt](B,D E,C)
  \end{scope}
  \begin{scope}[decoration={markings,
     mark=at position .20 with {\arrow[scale=2]{>}},
     mark=at position .70 with {\arrow[scale=2]{>}};}]
    \tkzDrawArc[postaction={decorate},color=black,line width=1pt](L,D)(E)
  \end{scope}
  \draw[green,thick] (0,0) circle (1.5cm);
  \draw[style=dashed] (0,0) circle (3);
  \draw (4.4,1) node[below left] {$\textcolor{red}{\widetilde{\mathcal{H}'}}$};
  \draw (5.4,6.5) node[below left] {$\textcolor{black}{|u|=1/q+\widetilde{\eta}}$};
  \draw (2.6,4) node[below left] {$\textcolor{black}{|u|=1/q}$};
  \draw (2.1,2.3) node[below left] {$\textcolor{green}{|u|=r}$};
\end{tikzpicture}
\end{figure}
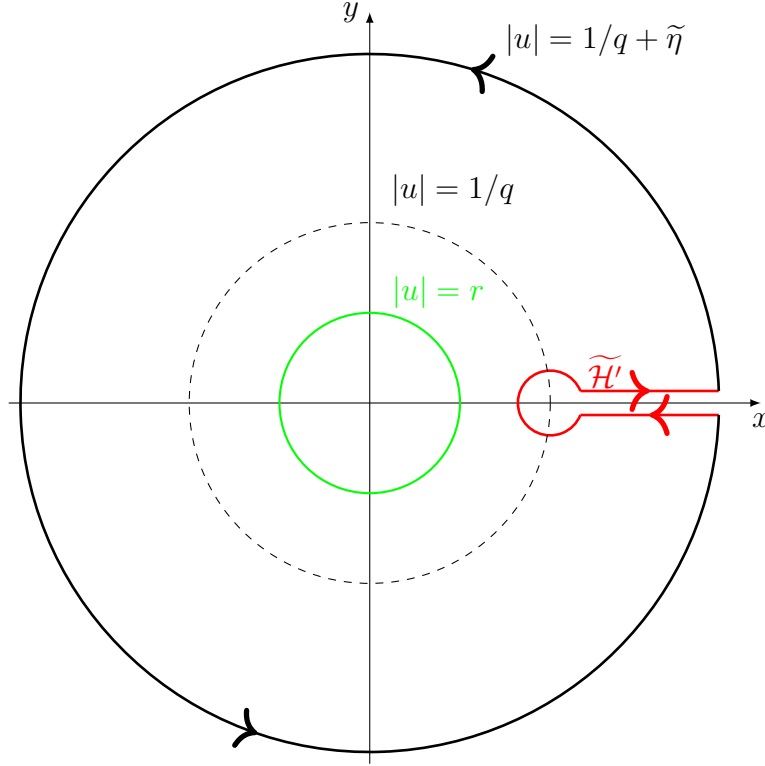
Now we again want to collect the contribution coming from the singularity of $\zeta(u)^z$ at $u=1/q$ to evaluate the above integral. We shift the contour $|u|=r$ to a bigger circle $|u|=1/q+\widetilde{\eta}$ with a different choice of $\widetilde{\eta}$ here (not the same choice for $\eta$ in Proposition \ref{la}) and with $|z|\leqslant q^2-\delta$. Since the region of absolute convergence for $\widetilde{F_z}(u)$ here is $\widetilde{\mathcal{R}}$, we have a different set of constraints on $\widetilde{\eta}$ than before. Here we get for $\widetilde{\eta}$, $|zu^2|\leqslant (q^{-1}+\widetilde{\eta})^2(q^2-\delta)\leqslant 1$ and $|u|=1/q+\widetilde{\eta}<q^{-\frac{1}{2}}$. This shows taking any $0<\widetilde{\eta} <\min\{1/\sqrt{q^2-\delta}-1/q,1/\sqrt{q}-1/q \}$ suffices. We then plan to use Lemma \ref{l3} to evaluate the integral in (\ref{lg}) with, $\widetilde{\eta}=\min\{1/\sqrt{q^2-\delta}-1/q,1/\sqrt{q}-1/q \}/2$ for $|z|\leqslant q^2-\delta$ just as we did in the proof of Proposition \ref{la}. \\
\begin{remark}
$F_z(u)$ was absolutely convergent in $\mathcal{R}$ which gave us the appropriate choice for $\eta$. All the calculations are identical here with $\widetilde{F_z}$ instead of $F_z$. The main difference is, here we have $|z|\leqslant q^2-\delta$ instead of $|z|\leqslant q-\epsilon$. Earlier ${F}_z\left(\frac{1}{q}\right)$ had a pole at $z=q,q^2,\ldots$ but we got rid of the $z=q$ pole here. This is what essentially allows us to take $|z|\leqslant q^2-\delta.$
\end{remark}
We therefore get
\begin{align}
\nonumber
\widetilde{M}_z(n) &= \dfrac{1}{2\pi i}\displaystyle\int\limits_{\{|u|=\frac{1}{q}+\widetilde{\eta\}}\cup\widetilde{\mathcal{H}'}} \zeta(u)^z\widetilde{F}_z(u)\dfrac{du}{u^{n+1}}\\
\nonumber
&=\widetilde{F}_z\left(\dfrac{1}{q}\right)q^n n^{z-1} \left( \dfrac{1}{\Gamma(z)}+O_{\delta}\left(\dfrac{1}{n}\right)\right)\\
\nonumber
&=q^nn^{z-1}\left(zh(z)+ O_{\delta}\left(\dfrac{1}{n}\right) \right),
\end{align}
where 
$$h(z)=\dfrac{1}{z}\widetilde{F}_z\left(\dfrac{1}{q}\right)\dfrac{1}{\Gamma(z)}=\dfrac{1}{\Gamma(z+1)}\left( 1-\dfrac{1}{q}\right)^{qz}\displaystyle\prod\limits_{p\in\mathcal{P}_{\geqslant 2}}\left(1-\frac{z}{q^{deg(p)}}\right)^{-1}\left( 1-\frac{1}{q^{deg(p)}}\right)^z.$$
Also we have $h(z)=\frac{1}{\Gamma(z+1)}\widetilde{F}_z\left(\frac{1}{q}\right)$. Since, $\widetilde{F}_z\left(\frac{1}{q}\right)$ is analytic in $|z|<q^2$, so is $h$.
\end{proof}
\subsection{An useful upper bound for $N'(n,j)$.}
\r*
\begin{proof}[Proof of Lemma \ref{l8}]\label{li}
We use Proposition \ref{lb} with $z=q+\eta$ and if $a_j(f)$ denotes the indicator function of polynomials in $\widetilde{M_n}$ with $\Omega(f)=j$ then, we see that every time $a_j=1$ occurs, it contributes an amount of $\left(q+\eta\right)^j$ on the left hand side of the sum of Proposition \ref{lb} and all the other terms are positive. We have by the previous lemma $h(z)$ is analytic in an open set containing the disc $|z|\leqslant q^2-\delta$. Since the disk is compact and $h$ is continuous there, we have $h(z)\ll_{q,\delta} 1$ for every $|z|\leqslant q^2-\delta$. Consequently for every $|z|\leqslant q^2-\delta$, $|zh(z)|\leqslant |(q+\eta)h(q+\eta)|\ll_{q,\delta} 1$. This observation leads to an upper bound $N'(n,j) \ll_{q,\delta} q^n n^{q+\eta-1}\left(\frac{1}{q+\eta}\right)^j$.
\end{proof}
\subsection{An uniform estimate for $N'(n,j)$ when $j\leqslant eq\log n$}\label{lj}
\s*
\begin{proof}[Proof of Lemma \ref{l11}]\label{X}
We use Proposition \ref{lb}. The coefficient of $z^j$ in $\widetilde{M}_z(n)$ is what we want to evaluate. We have for $r<q^2$,
$$N'(n,j)=\dfrac{1}{2\pi i}\displaystyle\int\limits_{|z|=r} \widetilde{M}_z(n) \dfrac{dz}{z^{j+1}}.$$
We recall from Proposition \ref{lb} $\widetilde{M}_z(n)=\dfrac{q^n}{n}\bigg \{ zh(z)n^z+O\left(n^{z-1}\right)\bigg \}.$
Plugging this back in the integral we get that
$$\dfrac{n}{q^n}N'(n,j)=\dfrac{1}{2\pi i}\displaystyle\int_{|z|=r}zh(z)n^z \dfrac{dz}{z^{j+1}}+O(E_0),$$
where $E_0$ is given by, $\displaystyle\int_{|z|=r} \left|\dfrac{n^{z-1}}{z^{j+1}}\right||dz|.$
At this point we choose $r=\frac{j}{\log n}\leqslant eq < q^2$ since by hypothesis $j\leqslant eq\log n$ and $q\geqslant 3$.\\
Noting $e^j=n^r\iff j=r\log n$ and using Lemma \ref{l16}, we get the term $E_0$ is bounded by
$$\ll \dfrac{r^{-j-1}e^j}{n}\displaystyle\int_{|z|=r} |n^{z-r}||dz|\ll \dfrac{r^{-j-1}e^j}{n}\dfrac{r}{j^{\frac{1}{2}}}\ll \dfrac{1}{n}\dfrac{(e\log n)^j}{j^{j+\frac{1}{2}}}\ll \dfrac{1}{n} \dfrac{(\log n)^j}{j!},$$
where the final step of bounding follows by Stirling.
This is the same trick that we used while bounding above $E$ on page 10 in the proof of Theorem \ref{l6}. 
We now return to the main term
$$M_0=\dfrac{1}{2\pi i}\displaystyle\int_{|z|=r}zh(z)n^z \dfrac{dz}{z^{j+1}}=\dfrac{1}{2\pi i}\displaystyle\int_{|z|=r}h(z)n^z \dfrac{dz}{z^j},$$
where $h(z)$ is an analytic function in the region $|z|<q^2$, so the integrand has a pole at $z=0$ of order $j$. We evaluate this integral using residue theorem and obtain
$$M_0=\lim\limits_{z\to 0}\dfrac{1}{(j-1)!}\dfrac{d^{j-1}}{dz^{j-1}} h(z)n^z.$$
Using binomial theorem for higher order differentiation of products of two complex functions
\begin{align}
\nonumber
\dfrac{1}{(j-1)!}\dfrac{d^{j-1}}{dz^{j-1}} h(z)n^z\Bigg|_{z=0} &=\dfrac{1}{(j-1)!}\displaystyle\sum\limits_{m+\ell=j-1}\dfrac{(j-1)!}{m!\ell!} h^{(m)}(z) (n^z)^{(\ell)}\Bigg|_{z=0}\\
\nonumber
&=\displaystyle\sum\limits_{m+\ell=j-1}\dfrac{1}{m!\ell!}h^{(m)}(0)(\log n)^\ell.
\end{align}
where we have taken the limits inside the argument of the higher order derivatives which is allowed because all the functions here are analytic and hence infinitely differentiable. Plugging everything back we obtain the desired result.
\end{proof}
\subsection{A technical upper bound for $h(z)$, $Q_j(X)$ and their derivatives.}
\u*
\begin{proof}[Proof of Lemma \ref{l12}]\label{lk}
From Proposition \ref{lb} we have that $h(z)$ is holomorphic for $|z|<q^2$ which in particular implies $h(z)\ll_{q} 1$ for $|z|\leqslant 2.9q$. Therefore we have by Cauchy's integral formula,
$$h^{(m)}(0)/m!=\dfrac{1}{2\pi}\left|\displaystyle\int\limits_{|z|=2.9q} h(z)\dfrac{dz}{z^{m+1}}\right|\ll \left( \dfrac{1}{2.9q}\right)^m \quad \left(2.9q<q^2\text{ whenever }q\geqslant 3\right).$$
We have from Lemma \ref{l13}
$$Q_j(X)=\displaystyle\sum\limits_{m+\ell=j-1}\dfrac{1}{m!\ell!}h^{(m)}(0) X^{\ell}\ll \displaystyle\sum\limits_{m+\ell=j-1} \left( \dfrac{1}{2.9q}\right)^m \dfrac{X^{\ell}}{\ell!}$$
We now take out the top term $\frac{X^{j-1}}{(j-1)!}$ to bound the sum from above and observe for each $m\geqslant 1$, $\frac{X^\ell}{\ell!}=\frac{X^{j-m-1}}{(j-m-1)!}\leqslant (eq)^m\frac{X^{j-1}}{(j-1)!}$ which is true after cross multiplying whenever $j\leqslant eqX$ since $ (j-1)(j-2)\cdots (j-m) \leqslant j^m \leqslant (eqX)^m$. Hence we arrive at the following upper bound,
$$Q_j(X)\ll \dfrac{X^{j-1}}{(j-1)!}\displaystyle\sum\limits_{0\leqslant m\leqslant j-1}  \left(\dfrac{eq}{2.9q}\right)^m\ll \dfrac{X^{j-1}}{(j-1)!}.$$
Likewise when $j\leqslant eqX,$
\begin{align}
\nonumber
Q_j'(X)&=\displaystyle\sum\limits_{m+\ell=j-1}\dfrac{1}{m!(\ell-1)!}h^{(m)}(0) X^{\ell-1}\\\
\nonumber
&\ll\displaystyle\sum\limits_{m+\ell=j-1} \left( \dfrac{1}{2.9q}\right)^m \dfrac{X^{\ell-1}}{(\ell-1)!}\\
\nonumber
&\ll \dfrac{X^{j-2}}{(j-2)!}\displaystyle\sum\limits_{0\leqslant m\leqslant j-1} \left(\dfrac{eq}{2.9q}\right)^m\\
\nonumber
&\ll \dfrac{X^{j-2}}{(j-2)!}.
\end{align}
\end{proof}
\subsection{One of the main ingredients of Proposition $\ref{l14}$ and Theorem $\ref{l10}$}
\t*
\begin{proof}[Proof of Lemma \ref{l13}]\label{lm}
\label{Proof of Lemma 4.4}
By Lemma \ref{l11} we have that
\begin{equation}\label{nw}
N'(n+j-k,j)=\dfrac{q^{n+j-k}}{n+j-k}\Bigg \{Q_j(\log (n+j-k))+ O \left( \dfrac{(\log (n+j-k))^j}{j!(n+j-k)}\right) \Bigg \}.
\end{equation}
We observe $n-k\geqslant n\left(1-\frac{1}{B}\right)\to\infty$ as $n\to\infty$. We also note if $E_1=\frac{(\log (n-k))^{j+1}}{j!(n-k)^2}$ and $E_2=\frac{(\log (n+j-k))^j}{j!(n+j-k)^2}$ then as $n\to\infty$
\begin{align}
\nonumber
\dfrac{E_2}{E_1} &= \dfrac{1}{(\log (n-k))}\left (\dfrac{\log (n+j-k)}{\log (n-k)} \right)^j \dfrac{(n-k)^2}{(n+j-k)^2}\\
\nonumber
&\leqslant\left( 1+ \dfrac{\log\left(1+\dfrac{j}{n-k}\right)}{\log(n-k)}\right)^j\dfrac{1}{(\log(n-k))}\to 0\\
\nonumber
&\implies E_2=o(E_1),
\end{align}
where we have at the second step the term inside the bracket is bounded, because $\frac{j}{n-k}\leqslant \frac{eq\log(n-k)}{(n-k)}\to 0$ and hence for $n$ large enough it is bounded by $(1+\log 2/\log(n-k))^{eq\log(n-k)} \leqslant e^{eq\log 2}=2^{eq}$. Thus from \eqref{nw} we arrive at
\begin{equation}\label{nw1}
    N'(n+j-k,j)=q^{n+j-k}\left \{\dfrac{Q_j(\log (n+j-k))}{n+j-k}+ O \left( E_1\right)\right \}
\end{equation}
We show at a small cost it is possible to replace $\frac{Q_j(\log (n+j-k))}{n+j-k}$ by $\frac{Q_j(\log (n-k))}{n-k}$.
We use mean value theorem to study the error.\\
Let $F(x):=\frac{Q_j(\log x)}{x}$ which is a regular function if $x>0$. The difference caused by this replacement
$$D=\dfrac{Q_j(\log (n+j-k))}{n+j-k}-\dfrac{Q_j(\log (n-k)}{n-k}=F(n+j-k)-F(n-k).$$
By mean value theorem one can pick $y\in (n-k,n+j-k)$ such that
$$D=jF'(y)=j\dfrac{Q_j'(\log y)-Q_j(\log y)}{y^2}=\dfrac{j}{y^2}Q_j'(\log y)-\dfrac{j}{y^2}Q_j(\log y)=F_1+F_2.$$
Therefore \eqref{nw1} implies
\begin{equation}\label{nw2}
     N'(n+j-k,j)=q^{n+j-k}\left \{\frac{Q_j(\log (n-k))}{n-k}+ F_1+F_2+O \left( E_1\right)\right \}
\end{equation}
We apply Lemma \ref{l12}  with the choice $X=\log y$. We note from the hypothesis that $j\leqslant eqY \leqslant eq\log y$ and $y>n-k$, so that we must have $X=\log y\to\infty$ with $n$. Using this observation, we upper bound the first term of $F_1$ which gives
$F_1=\frac{j}{y^2}Q_j'(\log y)\ll \frac{j}{y^2} \frac{(\log(n+j-k))^{j-2}}{(j-2)!}.$ Thus we obtain\\
\begin{align}
\nonumber
\dfrac{F_1}{E_1} &= \dfrac{(n-k)^2j!}{(\log(n-k))^{j+1}} \dfrac{j}{y^2} \dfrac{(\log(n+j-k))^{j-2}}{(j-2)!}\\
\nonumber
&\leqslant\left(\dfrac{\log(n+j-k)}{\log(n-k)}\right)^{j-2} \dfrac{j^2(j-1)}{(\log(n-k))^3}\\
\nonumber
&\ll 1,
\end{align}
since $j\leqslant eq \log(n-k)$ and $y>n-k.$
Hence $F_1=O(E_1)$ and likewise from Lemma \ref{l12} one can show
$|F_2|=\frac{j}{y^2}Q_j(\log y)\ll \frac{j(\log(n+j-k))^{j-1}}{y^2(j-1)!}.$\\
Thus we infer
\begin{align}
\nonumber
\dfrac{|F_2|}{E_1} &= \dfrac{j!(n-k)^2}{(\log(n-k))^{j+1}}\dfrac{j(\log(n+j-k))^{j-1}}{y^2(j-1)!}\\
\nonumber
&\leqslant \left( \dfrac{\log(n+j-k)}{\log(n-k)}\right)^{j-1}\dfrac{j^2}{(\log(n-k))^2}\\
\nonumber
&\ll 1.
\end{align}
The proof is now complete once we plug $F_1+F_2=O(E_1)$ in \eqref{nw2}.
\end{proof}

\end{document}